\documentclass[11pt, oneside]{article}

\usepackage{amssymb, amsthm, amsmath}

\usepackage{fullpage}

\usepackage{hyperref}

\newcommand{\N}{\mathbb{N}}
\newcommand{\Z}{\mathbb{Z}}

\newcommand{\R}{\mathbb{R}}
\newcommand{\C}{\mathbb{C}}

\newcommand{\calB}{\mathcal{B}}
\newcommand{\calE}{\mathcal{E}}
\newcommand{\calI}{\mathcal{I}}
\newcommand{\bfE}{\mathbf{E}}
\newcommand{\calZ}{\mathcal{Z}}

\renewcommand{\Re}{\operatorname{Re}}

\allowdisplaybreaks[3]

\theoremstyle{plain}
\newtheorem{theorem}{Theorem}[section]
\newtheorem{lemma}[theorem]{Lemma}
\newtheorem{proposition}[theorem]{Proposition}
\newtheorem{corollary}[theorem]{Corollary}

\theoremstyle{remark}
\newtheorem{remark}[theorem]{Remark}

\title{The computation of $\zeta(2k)$, $\beta(2k+1)$ and beyond by using telescoping series\thanks{To appear in ``Orthogonal Polynomials and Special Functions: In Memory of Jos\'e Carlos Petronilho'',
Coimbra Mathematical Texts book series, Springer.}}

\author{\'Oscar Ciaurri\textsuperscript{1}, Luis M. Navas\textsuperscript{2},
Francisco J. Ruiz\textsuperscript{3} and Juan L. Varona\textsuperscript{4}\\[6pt]
\small\textsuperscript{1}Departamento de Matem\'aticas y Computaci\'on,
Universidad de La Rioja,\\[-2pt]
\small 26006 Logro\~no, Spain. Email: {oscar.ciaurri@unirioja.es}\\[6pt]
\small\textsuperscript{2}Departamento de Matem\'aticas,
Universidad de Salamanca,\\[-2pt]
\small 37008 Salamanca, Spain. Email: {navas@usal.es}\\[6pt]
\small\textsuperscript{3}Departamento de Matem\'aticas,
Universidad de Zaragoza,\\[-2pt]
\small 50009 Zaragoza, Spain. Email: {fjruiz@unizar.es}\\[6pt]
\small\textsuperscript{4}Departamento de Matem\'aticas y Computaci\'on, 
Universidad de La Rioja,\\[-2pt]
\small 26006 Logro\~no, Spain. Email: {jvarona@unirioja.es}}

\date{}

\begin{document}

\maketitle

\begin{center}\itshape
Dedicated to the memory of Jos\'e Carlos Soares Petronilho
\bigskip
\end{center}

\begin{abstract}
We present some simple proofs of the well-known expressions for
\[
  \zeta(2k) = \sum_{m=1}^\infty \frac{1}{m^{2k}},
  \qquad
  \beta(2k+1) = \sum_{m=0}^\infty \frac{(-1)^m}{(2m+1)^{2k+1}},
\]
where $k = 1,2,3,\dots$, in terms of the Bernoulli and Euler polynomials.
The computation is done using only the defining properties of these polynomials 
and employing telescoping series. The same method also yields integral formulas 
for $\zeta(2k+1)$ and $\beta(2k)$.

In addition, the method also applies to series of type
\[
  \sum_{m\in\Z} \frac{1}{(2m-\mu)^s},
  \qquad
  \sum_{m\in\Z} \frac{(-1)^m}{(2m+1-\mu)^s},
\]
in this case using Apostol-Bernoulli and Apostol-Euler polynomials.

\smallskip

\end{abstract}

\section{Introduction}
\label{sec:intro}

Let us recall the definition of the Bernoulli and Euler polynomials, denoted by $B_{k}(x)$ and $E_{k}(x)$ respectively, in terms of the power series expansion of a generating function (see e.g.~\cite{Dil-NIST}). Namely,
\begin{equation}
\label{eq:polBerEu}
  \frac{ze^{xz}}{e^z-1} 
  = \sum_{k=0}^\infty B_k(x)\, \frac{z^k}{k!},
\qquad
  \frac{2 e^{xz}}{e^z+1}
  = \sum_{k=0}^\infty E_k(x)\, \frac{z^k}{k!}.
\end{equation}
These series are convergent in a neighborhood of $z = 0$.
It is not hard to show using~\eqref{eq:polBerEu} and some computation with Taylor series that $B_{k}(x)$ and $E_{k}(x)$ are monic polynomials of degree~$k$.


Consider now Riemann's zeta function and Dirichlet's beta function, defined by the series
\begin{equation}
\label{eq:dosseries}
  \zeta(s) = \sum_{m=1}^\infty \frac{1}{m^s},
  \qquad
  \beta(s) = \sum_{m=0}^\infty \frac{(-1)^m}{(2m+1)^s},
\end{equation}
for complex $s$ with $\Re(s) > 1$. Both may be analytically continued to the complex plane, although we are only interested in integer $s \geq 2$, so that we do not have to worry about this. The values
\[
  \zeta(2k) = \sum_{m=1}^\infty \frac{1}{m^{2k}},
  \qquad
  \beta(2k+1) = \sum_{m=0}^\infty \frac{(-1)^m}{(2m+1)^{2k+1}},
\]
may be expressed as follows in terms of the Bernoulli numbers $B_{2k} := B_{2k}(0)$ (or $B_{2k}(1)$, which is the same) and the Euler numbers $E_{2k} := 2^{2k}E_{2k}(1/2)$ (note that apart from the factor of $2^{2k}$, in the Euler case the evaluation is at $1/2$, not at~$0$). The exact formulas are:
\begin{equation}
\label{eq:sumber}
  \zeta(2k) = \frac{(-1)^{k-1} 2^{2k-1} \pi^{2k}}{(2k)!} B_{2k},
  \qquad k=1,2,3,\dots,
\end{equation}
\begin{equation}
\label{eq:sumeuler}
  \beta(2k+1) = \frac{(-1)^{k} \pi^{2k+1}}{2^{2k+2} (2k)!} \,E_{2k},
  \qquad k=0,1,2,\dots.
\end{equation}
These kinds of expressions have had great historical relevance. Euler was the first to prove~\eqref{eq:sumber}, in 1740, thus showing that $\zeta(2k)$ is always a rational multiple of $\pi^{2k}$. As for evaluation at odd positive integers, i.e. the values $\zeta(2k+1)$, this is much more mysterious and little is known about its arithmetical nature. In fact, the only established fact remains that $\zeta(3)$ is irrational, as proved by Ap\'ery in 1979.
In the case of the Dirichlet beta function, the reverse holds: \eqref{eq:sumeuler} expresses $\beta(2k+1)$ as a rational multiple of $\pi^{2k+1}$, but there are no known similar formulas for~$\beta(2k)$.

Throughout the years, many different proofs for~\eqref{eq:sumber} and \eqref{eq:sumeuler} have been found, varying in complexity and technical background necessary for their understanding. From the purpose of making them widely accessible, the ideal situation would be to be able to give simple, easy to understand proofs which are self-contained and do not use advanced machinery that one would need to study previously. A sample of various proofs, some of them only for $\zeta(2)$, may be found in~\cite{AiZi-zeta, AD, AmoDF, Ber, Duo, DwMi, Gr, Mor, Ri} and the bibliographies within those sources.

Many proofs of the formulas use the method of residues, or infinite products, or the pointwise convergence of Fourier series, which place them outside the scope of first year math students.
For example, let us assume that we know how to write Bernoulli and Euler polynomials in terms of trigonometric series, which were found by Adolf Hurwitz in 1890 (see \cite[\S\,24.8(i)]{Dil-NIST}). For even indices, we have
\begin{equation}
\label{eq:HurB2k}
  B_{2k}(x) = \frac{2(-1)^{k-1}(2k)!}{(2\pi)^{2k}} 
  \sum_{m=1}^{\infty} \frac{\cos(2\pi mx)}{m^{2k}},
  \quad
  x \in [0,1],
  \qquad k \geq 1.
\end{equation}
Now \eqref{eq:sumber} is obtained immediately by setting $x=0$. On the other hand, the corresponding Hurwitz expansion of the Bernoulli polynomials of odd index is 
\begin{equation}
\label{eq:HurB2k+1}
  B_{2k+1}(x) = \frac{2(-1)^{k-1}(2k+1)!}{(2\pi)^{2k+1}} 
  \sum_{m=1}^{\infty} \frac{\sin(2\pi mx)}{m^{2k+1}},
  \quad
  x \in [0,1],
  \qquad k \geq 1
\end{equation}
(also valid for $k=0$ and $0<x<1$). Now, however, \eqref{eq:HurB2k+1} is not useful for evaluating $\zeta(2k+1)$; for instance, if we take $x=0$ we only get the trivial expansion $0=0$. In the same way, the Euler polynomials satisfy 
\begin{equation}
\label{eq:HurE2k}
  E_{2k}(x) = \frac{4(-1)^{k}(2k)!}{\pi^{2k+1}} 
  \sum_{m=0}^{\infty} \frac{\sin((2m+1)\pi x)}{(2m+1)^{2k+1}},
  \quad
  x \in [0,1],
  \qquad k \geq 1;
\end{equation}
\begin{equation}
\label{eq:HurE2k-1}
  E_{2k-1}(x) = \frac{4(-1)^{k}(2k-1)!}{\pi^{2k}} 
  \sum_{m=0}^{\infty} \frac{\cos((2m+1)\pi x)}{(2m+1)^{2k}},
  \quad
  x \in [0,1],
  \qquad k \geq 1.
\end{equation}
Here \eqref{eq:sumeuler} arises from \eqref{eq:HurE2k} with $x=1/2$, but \eqref{eq:HurE2k-1} with $x=1/2$ gives $0=0$. The Hurwitz expansions \eqref{eq:HurB2k}--\eqref{eq:HurE2k-1} are Fourier expansions in terms of the usual trigonometric basis, and they can be easily obtained using standard methods for orthogonal systems, provided we quote the appropriate theorems for the pointwise convergence of Fourier series (they can also be obtained using the residue theorem, although this requires a course in complex analysis). We cannot avoid this background if we want to prove \eqref{eq:sumber} and \eqref{eq:sumeuler} as immediate consequences of the Hurwitz expansions.

No such theoretical background is needed for summing a telescoping series, whose convergence is usually straightforward to check. For example, the series $\sum_{n=1}^{\infty} 1/(n^2+n)$ is much easier to sum than $\sum_{n=1}^{\infty} 1/n^2$, since the partial fraction decomposition $1/(n^2+n) = 1/n - 1/(n+1)$ gives rise to a telescoping series.

In~\cite{Benko} it is shown how to compute the value $\zeta(2) = \pi^2/6$ by means of a telescoping series, and in~\cite{CNRV} this is done for all the values $\zeta(2k)$. The goal of this note is to give a similar procedure for computing $\beta(2k+1)$ by means of telescoping series; the idea is similar to the one on~\cite{CNRV} but, perhaps surprisingly, the fine tuning at some points is rather different. 

The proof we shall give only requires the use of Taylor series, without needing any additional knowledge of real or complex analysis. Since the Bernoulli and Euler numbers are frequently defined by means of generating functions which are Taylor series, it is natural to start from this basic requirement. Of course, there are other alternative ways to define them (see~\cite{CDG} for example). Thus, the aim of this paper is not to present new results, but rather to show how to sum certain series without invoking higher mathematics. As far as we know, the methods we use to achieve this are in fact new.


For completeness, and to be able to easily compare the techniques, we will also include the proof of \eqref{eq:sumber} that appears in our previous paper~\cite{CNRV} (and also in the textbook \cite[\S\,6.9]{Var}). In addition, we  also study series similar to \eqref{eq:dosseries}, namely
\begin{equation}
\label{eq:summus}
  \sum_{m\in\Z} \frac{1}{(2m-\mu)^s},
  \qquad
  \sum_{m\in\Z} \frac{(-1)^m}{(2m+1-\mu)^s},
\end{equation}
where the parameter $\mu$ is related to the parameter $\lambda$ in the generating functions
\begin{equation*}
  \frac{ze^{zx}}{\lambda e^z-1} = \sum_{k=0}^\infty \calB_k(x;\lambda)\, \frac{z^k}{k!},
  \qquad
  \frac{2e^{zx}}{\lambda e^z+1} = \sum_{k=0}^\infty \calE_k(x;\lambda)\, \frac{z^k}{k!}
\end{equation*}
which are used to define the Apostol-Bernoulli polynomials $\calB_k(x;\lambda)$ (they were introduced by Apostol in 1951 in~\cite{Ap-Lerch, Ap-Lerch-2}) and the Apostol-Euler polynomials $\calE_k(x;\lambda)$.

For the series which involve the Apostol-Bernoulli and Apostol-Euler polynomials, we will use complex exponentials rather than trigonometric functions, via Euler's formula $e^{i\theta} = \cos\theta+i\sin\theta$, con $\theta \in \R$. As is well-known, this tends to simplify calculations, as it happens often enough that the most efficient way of proving trigonometric identities is via conversion to complex exponentials. 
As far as Taylor series are concerned, the only difference is that now we use a complex variable and the domain of convergence is a complex disc rather than a real interval. 
However, a little care must be taken in manipulating the series related to Apostol-Bernoulli and Apostol-Euler polynomials, which may require a little more advanced knowledge, specifically regarding term-by-term limits and differentiation of infinite series of functions.

The organization of this paper is as follows. In Section~\ref{sec:propBE}, we list some properties of Bernoulli and Euler polynomials that will be useful in Sections~\ref{sec:sumB} and~\ref{sec:sumE}. In Section~\ref{sec:sumB} we prove \eqref{eq:sumber} by using a telescoping series.
In Section~\ref{sec:sumE} we prove \eqref{eq:sumeuler}, using a similar method but with enough differences to warrant attention. In Section~\ref{sec:ABAEsin} we give some properties of Apostol-Bernoulli and Apostol-Euler polynomials. In fact, these two families are essentially the same except for a change of variable in the parameter $\lambda$, so in the last sections we will only use Apostol-Euler polynomials. In Section~\ref{sec:mualt} we compute the series \eqref{eq:summus}(RHS) for $s=k=1,2,\dots$ in terms of Apostol-Euler polynomials evaluated at $x=1/2$. 
Section~\ref{sec:parfrac} we prove the expansion of $1/\sin^2(x)$ as an infinite sum of partial fractions, a result that will be used in Section~\ref{sec:musinalt}.
In Section~\ref{sec:musinalt} we compute the series \eqref{eq:summus}(LHS) for $s=k=2,3,\dots$ in terms of Apostol-Euler polynomials evaluated at $x=1$.

\begin{remark}
Another pair of functions related to the Riemann zeta function and also having a proper name are the Dirichlet eta and lambda functions, defined respectively by
\[
  \eta(s) = \sum_{m=1}^\infty \frac{(-1)^{m-1}}{m^s},
  \qquad
  \lambda(s) = \sum_{m=0}^\infty \frac{1}{(2m+1)^s}.
\]
It is easily checked that $\eta(s) = (1-2^{1-s})\zeta(s)$ and $\lambda(s) = (1-2^{-s})\zeta(s)$, so that determining formulas for their values reduces to the case of $\zeta(s)$.
\end{remark}

\section{Some basic properties of the Bernoulli and Euler polynomials}
\label{sec:propBE}

Starting from the definition~\eqref{eq:polBerEu} via generating functions and keeping in mind the uniqueness of power series representations, it is easy to prove many properties of the Bernoulli polynomials. We limit ourselves to those which we will actually need for our purpose. The literature contains a vast list of identities and properties (see for example~\cite{Dil-NIST}).

To begin with, multiplication of \eqref{eq:polBerEu}(LHS) by $(e^x - 1)/x$ leads to
\begin{equation*}
  x^n = \frac{1}{n+1} \sum_{k=0}^{n} \binom{n+1}{k} B_k(x), \qquad n = 0, 1, 2,\dots.
\end{equation*}
This shows that $B_0(x) = 1$ and allows us to recursively compute $B_{k}(x)$, proving by induction that $B_{k}(x)$ is a polynomial of degree $k$. The first few Bernoulli polynomials are
\[
  B_0(x) = 1,
  \qquad
  B_1(x) = x - \frac{1}{2},
  \qquad
  B_2(x) = x^2 - x + \frac{1}{6}.
\]
Differentiating~\eqref{eq:polBerEu}(LHS) with respect to $x$ and comparing $z$-coefficients, we obtain
\begin{equation}
\label{eq:derivBer}
  B'_k(x) = k B_{k-1}(x),
  \qquad k\geq 1.
\end{equation}
Another property which can be deduced from~\eqref{eq:polBerEu} is the symmetry relation
\begin{equation}
\label{eq:simetriaBer}
  B_k(1-x) = (-1)^k B_k(x),
  \qquad k\geq 0.
\end{equation}

Once the Bernoulli polynomials have been introduced, the Bernoulli \emph{numbers} may be defined by $B_k = B_k(0)$ (equivalently, by setting $x=0$ in the generating function~\eqref{eq:polBerEu}). It is easily checked that $x/(e^x-1) + x/2$ is an even function, hence only even powers appear in its Taylor series. From this we see that $B_1 = -1/2$ and $B_{2k+1} = 0$ for all $k \in \N$. Using~\eqref{eq:simetriaBer} we conclude that
\begin{equation}
\label{eq:odd0}
  B_{2k+1}(1) = B_{2k+1}(0) = 0, \qquad k \ge 1.
\end{equation}
It is also true by~\eqref{eq:simetriaBer} that $B_{2k}(1) = B_{2k}(0)$, but these are nonzero. In any case, by~\eqref{eq:derivBer} we have
\begin{equation}
\label{eq:intBer}
  \int_0^{1} B_{k}(x)\,dx = \frac{1}{k+1} (B_{k+1}(1)-B_{k+1}(0)) = 0, \qquad k \ge 1.
\end{equation}

The Euler polynomials have quite similar properties, which one may arrive at in essentially the same way as we have described for Bernoulli polynomials, in this case starting from~\eqref{eq:polBerEu}. The first few Euler polynomials are
\[
  E_0(x) = 1,
  \qquad
  E_1(x) = x - \frac{1}{2},
  \qquad
  E_2(x) = x^2 - x.
\]
They satisfy
\begin{equation}
\label{eq:derivEuler}
  E'_k(x) = k E_{k-1}(x),
  \qquad k\geq 1,
\end{equation}
and
\begin{equation*}
  E_{2k}(0) = E_{2k}(1) = 0,
  \qquad k\geq 1.
\end{equation*}

\section{The value of $\zeta(2k)$ using $B_{2k}(x)$}
\label{sec:sumB}

Let us start with some auxiliary integrals:

\begin{lemma}
\label{lem:Ikm}
For $k \ge 0$ and $m \ge 1$, let
\[
  I_{k,m} = \int_0^{1} B_{2k}(x) \cos(m\pi x)\,dx.
\]
Then
\begin{equation}
\label{eq:evalInt}
  I_{k,m} = 
  \begin{cases}
  0, & m=1,3,5,\dots, \\
  \displaystyle \frac{(-1)^{k-1}(2k)!}{m^{2k} \pi^{2k}}, & m=2,4,6,\dots.
  \end{cases}
\end{equation}
\end{lemma}

\begin{proof}
Begin by noting that
\[
  I_{0,m} = \int_0^{1} \cos(m\pi x)\,dx = 0, \quad m= 1,2,3,\dots.
\]
For $k \ge 1$, integrate by parts twice, applying~\eqref{eq:derivBer} both times.
Since $\sin(m \pi x)$ vanishes at $t=0$ and~$t=1$, we obtain
\begin{align*}
  I_{k,m} &= \int_0^{1} B_{2k}(x) \cos(m\pi x)\,dx \\
  &= \frac{1}{m\pi} \Big(B_{2k}(x) \sin(m \pi x) \Big)\Big|_{x=0}^{1}
  - \frac{2k}{m\pi} \int_0^{1} B_{2k-1}(x) \sin(m\pi x)\,dx \\
  &= 0 - \frac{2k}{m\pi} \int_0^{1} B_{2k-1}(x) \sin(m\pi x)\,dx \\
  &= \frac{2k}{m^{2} \pi^{2}} \Big(B_{2k-1}(x) \cos(m \pi x) \Big)\Big|_{x=0}^{1}
  - \frac{2k(2k-1)}{m^2\pi^2} \int_0^{1} B_{2k-2}(x) \cos(m\pi x)\,dx. 
\end{align*}
For $k = 1$, the two polynomials on the last line are $B_{1}(x) = x - 1/2$ and $B_{0}(x) = 1$, which yields the special case
\begin{align*}
  I_{1,m} &= \int_0^{1} B_2(x) \cos(m\pi x)\,dx \\
  &= \frac{2}{m^{2} \pi^{2}} \Big((x-1/2) \cos(m \pi x) \Big)\Big|_{x=0}^{1}
  - \frac{2}{m^2\pi^2} \int_0^{1} \cos(m\pi x)\,dx \\
  &= \begin{cases}
  0, & m=1,3,5,\dots, \\
  \dfrac{2}{m^2 \pi^2}, & m=2,4,6,\dots.
  \end{cases}
\end{align*}
For $k \geq 2$, on the other hand, $B_{2k-1}(x)$ vanishes at $t=0$ and~$t=1$ (by~\eqref{eq:odd0}), from which we obtain the recurrence relation
\[
  I_{k,m} = - \frac{2k(2k-1)}{m^{2} \pi^{2}} \,I_{k-1,m},
  \qquad
  k \geq 2,
\]
and an easy induction in $k$ yields~\eqref{eq:evalInt}.
\end{proof}

On our way to summing $\sum_{n=1}^{\infty} 1/n^{2k}$, we shall also need the following trigonometric identity
\begin{equation}
\label{eq:idtrig}
  \cos(mt) = \frac{\sin(\frac{2m+1}{2}t)-\sin(\frac{2m-1}{2}t)}{2\sin(\frac{t}{2})},
\end{equation}
which can be proved by elementary trigonometry. It will be used to obtain a telescoping series.

We also note the following lemma which will be used several times.

\begin{lemma}
\label{lem:RL}
Let $f$ be a continuously differentiable function on $[0, 1]$. Then
\[
\lim_{R \to \infty} \int_0^{1} f(x) \sin(R x)\,dx = 0.
\]
\end{lemma} 

\begin{proof}
Integrating by parts,
\[
  \int_0^1 f(x) \sin(Rx)\,dx 
  = - \frac{\cos(R)}{R} f(1) + \frac{1}{R} f(0) + \int_0^1 f'(x) \frac{\cos(Rx)}{R} \,dx.
\]
Since $f'(x)$ is bounded, as is the cosine of course, each term tends to $0$ when $R \to \infty$, hence so does their sum.
\end{proof}

Keeping the previous facts in mind, we turn to the main result of this section.

\begin{theorem}[Euler, 1740]
\label{teo:zeta2k}
For every positive integer $k$ we have
\begin{equation}
\label{eq:zeta2k}
  \zeta(2k) = \sum_{n=1}^\infty \frac{1}{n^{2k}} = \frac{(-1)^{k-1} (2\pi)^{2k}}{2(2k)!} \,B_{2k},
\end{equation}
where $B_{2k}$ are the Bernoulli numbers of even index.
\end{theorem}

\begin{proof}
To have polynomials that vanish at $x=0$ (the reason is made clear below), 
let $b_{2k}(x) = B_{2k}(x)-B_{2k}(0)$ and take
\[
  I'_{k,m} = \int_0^{1} b_{2k}(x) \cos(m\pi x)\,dx
  = \int_0^{1} (B_{2k}(x)-B_{2k}) \cos(m\pi x)\,dx,
  \qquad k \ge 0, \quad m \ge 1.
\]
Since $\int_0^1 \cos(m\pi x)\,dx = 0$ for every positive integer $m$, it is clear that $I'_{k,m} = I_{k,m}$, so again
\[
  I'_{k,m} = 
  \begin{cases}
  0, & m=1,3,5,\dots, \\
  \displaystyle \frac{(-1)^{k-1}(2k)!}{m^{2k} \pi^{2k}}, & m=2,4,6,\dots.
  \end{cases}
\]

Using that $I'_{k,2m+1} = 0$ for every $m$, applying the trigonometric identity~\eqref{eq:idtrig} and canceling terms in the resulting telescoping series, we have
\begin{multline*}
  \frac{(-1)^{k-1}(2k)!}{\pi^{2k}} \sum_{m=1}^{\infty} \frac{1}{(2m)^{2k}} 
  = \sum_{m=1}^{\infty} I'_{k,2m}
  = \sum_{m=1}^{\infty} I'_{k,2m} + \sum_{m=0}^{\infty} I'_{k,2m+1} \\
\begin{aligned}
  &= \sum_{m=1}^{\infty} I'_{k,m}
  = \sum_{m=1}^{\infty} \int_0^{1} b_{2k}(x) \cos(m\pi x)\,dx \\
  &= \lim_{N \to \infty} \sum_{m=1}^{N} \Bigg( 
    \int_0^{1} b_{2k}(x) \frac{\sin(\frac{2m+1}{2}\pi x)}{2\sin(\frac{\pi x}{2})}\,dx
    - \int_0^{1} b_{2k}(x) \frac{\sin(\frac{2m-1}{2}\pi x)}{2\sin(\frac{\pi x}{2})}\,dx \Bigg) \\
  &= \lim_{N \to \infty} \int_0^{1} b_{2k}(x) \frac{\sin(\frac{2N+1}{2}\pi x)}{2\sin(\frac{\pi x}{2})}\,dx
  - \int_0^{1} b_{2k}(x) \frac{\sin(\frac{\pi x}{2})}{2\sin(\frac{\pi x}{2})}\,dx.
\end{aligned}
\end{multline*}
Let us look closely at the last line. We begin by checking that
\begin{equation}
\label{eq:limBer0}
  \lim_{N \to \infty} \int_0^{1} b_{2k}(x) 
  \frac{\sin(\frac{2N+1}{2}\pi x)}{2\sin(\frac{\pi x}{2})}\,dx = 0.
\end{equation}
Indeed, the function
\[
  f(x) = \frac{b_{2k}(x)}{2\sin(\frac{\pi x}{2})}, \qquad x \in (0,1],
\] 
extends by continuity to $x=0$ (note that $b_{2k}(0)=0$ so that $b_{2k}(x)$ is divisible by $x$) and is continuously differentiable on $[0,1]$.
Setting $R = (2N + 1)\pi/2$, Lemma~\ref{lem:RL} implies~\eqref{eq:limBer0}.

Using \eqref{eq:intBer} and recalling that $B_{2k}(0)=B_{2k}$, the value of the last integral is
\begin{align*}
  \int_0^{1} b_{2k}(x) \frac{\sin(\frac{\pi x}{2})}{2\sin(\frac{\pi x}{2})}\,dx
  &= \frac{1}{2} \int_0^{1} (B_{2k}(x)-B_{2k}) \,dx 
  = -\frac{1}{2} \int_0^{1} B_{2k} \, dx 
  = -\frac{B_{2k}}{2}.
\end{align*}
Thus, finally
\[
  \frac{(-1)^{k-1}(2k)!}{2^{2k}\pi^{2k}} \sum_{m=1}^{\infty} \frac{1}{m^{2k}} = 0 + \frac{B_{2k}}{2}
\]
and we have proved Euler's formula for $\zeta(2k)$.
\end{proof}

Readers familiar with Fourier series will have recognized that the integrals $I_{k,2m}$ evaluated in Lemma~\ref{lem:Ikm} are the Fourier coefficients of the even index Bernoulli polynomials $B_{2k}(x)$, whose Fourier series is~\eqref{eq:HurB2k}.
As we have mentioned in the introduction, the Fourier series yields a quick proof of
~\eqref{eq:zeta2k}, the drawback being that one would need to introduce Fourier series and in particular their pointwise convergence. The proof presented in Theorem~\ref{teo:zeta2k} only uses a telescoping series which is easy to sum and avoids all but the basic theory of the convergence of infinite series. 

It is also worth noting that Lemma~\ref{lem:RL} is the special case for continuously differentiable functions of the Riemann-Lebesgue lemma, whose proof only requires one-variable calculus.

\begin{remark}
It is tempting to ask ourselves if the above method, or a similar one, can be used to gain some insight into the values $\zeta(2k+1)$. The answer is a (very) qualified ``yes''. If we take
\[
  J_{k,m} = \int_0^{1} B_{2k+1}(x) \sin(m\pi x)\,dx
\]
and proceed as in Lemma~\ref{lem:Ikm}, we obtain
\[
  J_{k,m} = 
  \begin{cases}
  0, & m=1,3,5,\dots, \\
  \dfrac{(-1)^{k-1}(2k+1)!}{m^{2k+1} \pi^{2k+1}}, & m=2,4,6,\dots.
  \end{cases}
\]
Using the corresponding trigonometric identity
\[
  \sin(mt) = -\frac{\cos(\frac{2m+1}{2}t)-\cos(\frac{2m-1}{2}t)}{2\sin(\frac{t}{2})},
\]
and continuing as in the proof of Theorem~\ref{teo:zeta2k}, does indeed yield a telescoping series. The limit in $N$ is still null by Lemma~\ref{lem:RL}, but the final integral is not elementary. The final result is
\[
  \zeta(2k+1) = \frac{(-1)^{k-1}2^{2k}\pi^{2k+1}}{(2k+1)!}
  \int_0^{1} B_{2k+1}(x) \frac{\cos(\frac{\pi x}{2})}{\sin(\frac{\pi x}{2})}\,dx,
  \quad k = 1,2,3,\dots.
\]
Thus the series $\sum_{n=1}^{\infty} 1/n^{2k+1}$ is transformed into an integral, the nature of its value remaining elusive.
\end{remark}

\section{The sum of the alternating series $\beta(2k+1)$ using $E_{2k}(x)$}
\label{sec:sumE}

Once again, we begin with some auxiliary integrals.

\begin{lemma}
\label{lem:IkmEuler}
For $k \ge 0$ and $m \ge 1$, let
\[
  I_{k,m} = \int_0^{1} E_{2k}(x) \sin(m\pi x)\,dx.
\]
Then
\begin{equation}
\label{eq:evalIntEuler}
  I_{k,m} = 
  \begin{cases}
  \displaystyle \frac{2(-1)^{k}(2k)!}{m^{2k+1} \pi^{2k+1}}, & m=1,3,5,\dots, \\
  0, & m=2,4,6,\dots. 
  \end{cases}
\end{equation}
\end{lemma}

\begin{remark}
It is easy to verify that the integrals $I_{k,2m} = \int_0^{1} E_{2k}(x) \sin(2m\pi x)\,dx$ are all zero for $m \ge 1$; however, as opposed to what happens in the proof of Theorem~\ref{teo:zeta2k} in the case of the Riemann zeta function, these integrals do not appear in the present proof. 
\end{remark}

\begin{proof}
For $k=0$ we can immediately compute
\[
  I_{0,m} = \int_0^{1} \sin(m\pi x)\,dx = -\frac{(-1)^m-1}{m\pi}
  =
  \begin{cases}
  2/(m\pi), & m= 1,3,5,\dots,\\
  0, & m= 2,4,6,\dots.
  \end{cases}
\]
For $k \ge 1$, we integrate by parts twice, applying~\eqref{eq:derivEuler} at each step. Since $E_{2k}(x)$ vanishes at $t=0$ and~$t=1$, as does $\sin(m \pi x)$, we obtain
\begin{align*}
  I_{k,m} &= \int_0^{1} E_{2k}(x) \sin(m\pi x)\,dx \\
  &= \frac{-1}{m\pi} \Big(E_{2k}(x) \cos(m \pi x) \Big)\Big|_{x=0}^{1}
  + \frac{2k}{m\pi} \int_0^{1} E_{2k-1}(x) \cos(m\pi x)\,dx \\
  &= 0 + \frac{2k}{m\pi} \int_0^{1} E_{2k-1}(x) \cos(m\pi x)\,dx \\
  &= \frac{2k}{m^{2} \pi^{2}} \Big(E_{2k-1}(x) \sin(m \pi x) \Big)\Big|_{x=0}^{1}
  - \frac{2k(2k-1)}{m^2\pi^2} \int_0^{1} E_{2k-2}(x) \sin(m\pi x)\,dx \\
  &= 0 - \frac{2k(2k-1)}{m^2\pi^2} \int_0^{1} E_{2k-2}(x) \sin(m\pi x)\,dx
  = - \frac{2k(2k-1)}{m^2\pi^2} \, I_{k-1,m}.
\end{align*}
This proves \eqref{eq:evalInt} by induction in~$k$.
\end{proof}

Our technique for summing $\sum_{m=0}^{\infty} (-1)^{m}/(2m+1)^{2k+1}$ uses the Euler polynomials, as well as the trigonometric identity
\begin{equation}
\label{eq:idtrigEuler}
  \sin((2m+1)t) = \frac{\sin((2m+2)t)+\sin(2mt)}{2\cos(t)},
\end{equation}
which is again easy to prove by basic trigonometry, and will allow us to transform
$\sum_{m=0}^{\infty} (-1)^{m}/(2m+1)^{2k+1}$ into a telescoping series.

We also need the Taylor series of the arctangent,
\begin{equation}
\label{eq:arctan}
  \arctan(t) = \sum_{m=0}^{\infty} \frac{(-1)^m}{2m+1} \,t^m,
  \qquad -1 < t \le 1.
\end{equation}
In particular, $\sum_{m=0}^{\infty} (-1)^{m}/(2m+1) = \arctan(1) = \pi/4$.
Since $\beta(1) = \sum_{m=0}^{\infty} (-1)^{m}/(2m+1)$ and $E_0=1$, this proves the case $k=0$ of~\eqref{eq:sumeuler}. Thus from now on we may assume that $k \ge 1$.
Note that in the proof of Theorem~\ref{teo:zeta2k} we did not need to use \eqref{eq:arctan} or any similar expression.

With the setup complete, we turn to the main result of this section. 

\begin{theorem}
\label{teo:beta2k+1}
For each positive integer $k$, we have
\begin{equation}
\label{eq:beta2k+1}
  \beta(2k+1) = \sum_{m=0}^\infty \frac{(-1)^m}{(2m+1)^{2k+1}} 
  = \frac{(-1)^{k} \pi^{2k+1}}{2^{2k+2} (2k)!} \,E_{2k},
\end{equation}
where $E_{2k} = 2^{2k}E_{2k}(1/2)$ denotes the Euler numbers of even index.
\end{theorem}

\begin{proof}
Let $e_{2k}(x) = E_{2k}(x)-E_{2k}(1/2)$. Since
\[
  \int_0^1 E_{2k}(1/2) \sin((2m+1)\pi x) \,dx = \frac{2E_{2k}(1/2)}{(2m+1)\pi},
\]
we have
\[
  I'_{k,2m+1} := \int_0^1 e_{2k}(x)\sin((2m+1)\pi x)\,dx 
  = I_{k,2m+1} - \frac{2E_{2k}(1/2)}{(2m+1)\pi}.
\]
Now, by~\eqref{eq:evalIntEuler},
\begin{align*}
  & \frac{2(-1)^{k}(2k)!}{\pi^{2k+1}} \sum_{m=0}^{\infty} \frac{(-1)^m}{(2m+1)^{2k+1}} 
  = \sum_{m=0}^{\infty} (-1)^m I_{k,2m+1} \\
  &\qquad\qquad
  = \sum_{m=0}^{\infty} (-1)^m I'_{k,2m+1} + \frac{2E_{2k}(1/2)}{\pi} \sum_{m=0}^{\infty} \frac{(-1)^m}{2m+1}.
\end{align*}
We already know that $\sum_{m=0}^{\infty} \frac{(-1)^m}{2m+1} = \pi/4$, so let us take a look at $\sum_{m=0}^{\infty} (-1)^m I'_{k,2m+1}$.

Applying the trigonometric identity~\eqref{eq:idtrigEuler} and canceling terms in the resulting telescoping series, we obtain the following (note that the denominator $\cos(\pi x)$ only vanishes at $x=1/2$, where $e_{2k}(x)$ does so as well):
\begin{align*}
  &\sum_{m=0}^{\infty} (-1)^m I'_{k,2m+1} 
  = \sum_{m=0}^{\infty} (-1)^m \int_0^{1} e_{2k}(x) \sin((2m+1)\pi x)\,dx \\
  &\qquad= \lim_{N \to \infty} \sum_{m=0}^{N} (-1)^m \Bigg( 
    \int_0^{1} e_{2k}(x) \frac{\sin((2m+2)\pi x)}{2\cos(\pi x)}\,dx
    + \int_0^{1} e_{2k}(x) \frac{\sin(2m\pi x)}{2\cos(\pi x)}\,dx \Bigg) \\
  &\qquad= \lim_{N \to \infty} \sum_{m=0}^{N} \Bigg(
    (-1)^{m} \int_0^{1} e_{2k}(x) \frac{\sin((2m+2)\pi x)}{2\cos(\pi x)}\,dx
    - (-1)^{m-1} \int_0^{1} e_{2k}(x) \frac{\sin(2m\pi x)}{2\cos(\pi x)}\,dx \Bigg) \\
  &\qquad= \lim_{N \to \infty} (-1)^N \int_0^{1} e_{2k}(x) \frac{\sin((2N+2)\pi x)}{2\cos(\pi x)}\,dx
    + \int_0^{1} e_{2k}(x) \frac{\sin(0 \cdot \pi x)}{2\cos(\pi x)}\,dx.
\end{align*}
The last term is obviously null.

Let us check that
\[
  \lim_{N \to \infty} \int_0^{1} e_{2k}(x) \frac{\sin((2N+2)\pi x)}{2\cos(\pi x)}\,dx = 0.
\]
The function
\[
  f(x) = \frac{e_{2k}(x)}{2\cos(\pi x)}, \qquad x \in [0,1] \setminus \{1/2\},
\] 
extends by continuity to $x=1/2$ (recall that $e_{2k}(0)=0$) and is continuously differentiable on $[0,1]$. Setting $R = (2N + 2)\pi$ and applying Lemma~\ref{lem:RL} we conclude that the limit as $N \to \infty$ is also~$0$.

Thus we have proved that
\[
\frac{2(-1)^{k}(2k)!}{\pi^{2k+1}} \sum_{m=0}^{\infty} \frac{(-1)^m}{(2m+1)^{2k+1}}
= 0 + 0 + \frac{2E_{2k}(1/2)}{\pi} \frac{\pi}{4},
\]
hence
\[
  \beta(2k+1) = \frac{(-1)^{k}\pi^{2k+1}}{2(2k)!} \frac{E_{2k}(1/2)}{2} 
  = \frac{(-1)^{k}\pi^{2k+1} E_{2k}}{2^{2k+2} (2k)!},
\]
which concludes the proof.
\end{proof}

As in the proof of the formula for $\zeta(2k)$, this method contains thinly disguised Fourier analysis. The integrals $I_{k,2m+1}$ evaluated in Lemma~\ref{lem:IkmEuler} are the Fourier coefficients of the even index Euler polynomials $E_{2k}(x)$, whose Fourier series is \eqref{eq:HurE2k}. Evaluating this series at $x = 1/2$ yields~\eqref{eq:beta2k+1} immediately, but our proof does not require any knowledge of this, as again a convenient telescoping series may be found that achieves what we want.


\begin{remark}
What happens if we use this method to try to evaluate $\beta(2k)$? The situation is analogous to what happened with $\zeta(2k+1)$. Taking
\[
  J_{k,m} = \int_0^{1} E_{2k+1}(x) \cos(m\pi x)\,dx
\]
and proceeding as in Lemma~\ref{lem:IkmEuler} we obtain
\[
  J_{k,m} = 
  \begin{cases}
  0, & m=0,2,4,6\dots,\\
  \dfrac{2(-1)^{k-1}(2k+1)!}{m^{2k+2} \pi^{2k+2}}, & m=1,3,5,\dots.
  \end{cases}
\]
Using the analogous trigonometric identity
\[
  \cos((2m+1)t) = \frac{\cos((2m+2)t)+\cos(2mt)}{2\cos(t)}, 
\]
and continuing as in the proof of Theorem~\ref{teo:beta2k+1}, also yields a telescoping series, with null limit in $N$, but the remaining integral is not elementary. The analogous formula is
\[
  \beta(2k+2) = \frac{(-1)^{k-1}\pi^{2k+2}}{4(2k+1)!}
  \int_0^{1} \frac{E_{2k+1}(x)}{\cos(\pi x)}\,dx,
  \quad k = 0,1,2,\dots,
\]
but again, this does not provide immediate information about these values.
\end{remark}

\section{Apostol-Bernoulli and Apostol-Euler polynomials}
\label{sec:ABAEsin}

Our aim now is to study series of a more general form than~\eqref{eq:dosseries}. We add a real parameter $\mu \in \R$ and consider the series~\eqref{eq:summus}.  
A few remarks are in order. First, we can avoid values of $\mu$ that lead to the vanishing of a denominator in the general term of the series, or we can suppress the unique summand where that happens.
Note also that the index of summation in~\eqref{eq:summus} runs over the complete set of integers $\Z$, not just over $\N$, as was the case with~\eqref{eq:dosseries}. Straightforward manipulation transforms a series of type \eqref{eq:summus} into one where the index of summation runs over $\N$, especially in easy cases such as $\mu = 0$ or in general $\mu \in \Z$, removing the summand where a denominator vanishes. In particular this shows that the series in~\eqref{eq:dosseries} are special cases of those in~\eqref{eq:summus}. We will not go into more detail on these points at the moment. 

Just as we used Bernoulli and Euler polynomials to some the series~\eqref{eq:dosseries}, we shall now use Apostol-Bernoulli and Apostol-Euler polynomials to sum~\eqref{eq:summus}.
We recall that these polynomials are defined respectively by
\begin{equation*}
  \frac{z e^{xz}}{\lambda e^z-1} 
  = \sum_{k=0}^\infty \calB_k(x;\lambda)\, \frac{z^k}{k!},
  \qquad \lambda \in \C,
\end{equation*}
and
\begin{equation}
\label{eq:GF-AE}
  \frac{2 e^{zx}}{\lambda e^z+1}
  = \sum_{k=0}^\infty \calE_k(x;\lambda)\, \frac{z^k}{k!},
  \qquad \lambda \in \C \setminus \{-1\}.
\end{equation}
The case $\lambda=0$ is trivial and will not be considered, while for $\lambda=1$ the Apostol-Bernoulli polynomials $\calB_k(x;1)$ coincide with the Bernoulli polynomials $B_k(x)$, and the Apostol-Euler polynomials $\calE_k(x;1)$ are the Euler polynomials $E_k(x)$.

In addition, it follows easily from the definition that
\begin{equation}
\label{eq:relAB-AE}
  \calE_{k}(x;\lambda) = - \frac{2}{k+1} \calB_{k+1}(x;-\lambda);
\end{equation}
so that it can be said that these are basically the same polynomial family except for a sign change in the parameter, a shift in the index, and a scaling factor. Apparently, even though~\eqref{eq:relAB-AE} is not hard to discover, no one seems to have done so before~\cite{NRV}; more details are given in~\cite{NRV-Brno}.

In any case, it suffices to consider either only $\calB_k(x;\lambda)$ or only $\calE_k(x;\lambda)$ .
Now, excepting the case $\lambda=1$, $\calB_k(x;\lambda)$ has degree $k-1$, so it is more convenient to use $\calE_{k}(x;\lambda)$ which has degree~$k$.

It is important to note that~\eqref{eq:GF-AE} is not well-defined for $\lambda=-1$, since
${2 e^{zx}}/(-e^z+1)$ has a pole at $z=0$ and hence does not have a Taylor expansion (one would need to talk about Laurent expansions). Thus we can't simply study the case $\calB_{k}(x;1)$ using $\calE_{k}(x;-1)$. This is not a problem since $\calB_{k}(x;1) = B_{k}(x)$ has already been dealt with in Section~\ref{sec:sumB}.

Although our goal now is to evaluate series of the type~\eqref{eq:summus} using only the definitions of the corresponding polynomials via generating functions and coming up with a suitable telescoping series,  we may point out that the Apostol-Bernoulli and Apostol-Euler polynomials have Fourier series expansions similar to those of the Bernoulli and Euler polynomials. Indeed (see~\cite{NRV}), the Fourier series of the Apostol-Bernoulli polynomials is
\begin{equation*}
  \calB_k(x;\lambda) = -\delta_k(x;\lambda) - \frac{k!}{\lambda^x} \sum_{m \in \Z \setminus\{0\}} 
  \frac{e^{2\pi i m x}}{(2\pi i m - \log \lambda)^k},
  \qquad 0 < x < 1,
\end{equation*}
where $\delta_k(x;\lambda) = 0$ or $\frac{(-1)^k k!}{\lambda^x \log^k \lambda}$ according as $\lambda = 1$ or $\lambda \ne 1$. The case $\lambda=1$ corresponds to the Bernoulli polynomials. Since we will only consider the case $\lambda\ne1$, we can rewrite it as
\begin{equation}
\label{eq:FS-AB}
  \calB_k(x;\lambda) = - k! \sum_{m \in \Z} 
  \frac{e^{(2\pi i m - \log\lambda) x}}{(2\pi i m - \log\lambda)^k},
  \qquad 0 < x < 1.
\end{equation}
The corresponding Fourier series of the Apostol-Euler polynomials is
\begin{equation}
\label{eq:FS-AE}
\begin{aligned}
    \calE_k(x;\lambda) 
  &= \frac{2 \cdot k!}{\lambda^x} \sum_{m \in \Z} 
    \frac{e^{(2m+1)\pi i x}}{((2m+1)\pi i - \log \lambda)^{k+1}} \\
  &= 2 \cdot k! \sum_{m \in \Z} 
    \frac{e^{((2m+1)\pi i - \log\lambda) x}}{((2m+1)\pi i - \log\lambda)^{k+1}},
  \qquad 0 < x < 1,
\end{aligned}
\end{equation}
with $\lambda = 1$ corresponding to the Euler polynomials.

In order to link the polynomials with parameter $\lambda$ to the series with parameter $\mu$ we need to let $\lambda=e^{i\mu}$, so that $\log\lambda = i\mu$, with $\mu \in \R$. 

As was the case before, determining the values of the series \eqref{eq:summus} reduces to substituting a specific value of $x$ in the above Fourier series. To reiterate, although we mention this as motivation, our procedure is to avoid this by finding elementary methods based on telescoping series.

The Apostol-Bernoulli and Apostol-Euler polynomials satisfy a number of properties similar to those of the Bernoulli and Euler polynomials, some of which were used in Section~\ref{sec:propBE}. Indeed, these are all examples of Appell sequences, which satisfy a number of universal relations (see \cite{NRV, NRV18, NRV19, NRV-Brno}). For example,
\[
  \frac{d}{dx} \calE_k(x;\lambda) = k \calE_{k-1}(x;\lambda).
\]
Instead of making a list of properties here to be used later on, we will introduce them as needed.

\section{Apostol-Euler polynomials to sum an ``alternating series''}
\label{sec:mualt}

Setting $\lambda = e^{\mu i}$, $\mu \in \R$, we want to find a closed formula for the sum
\begin{equation}
\label{eq:calZ}
  \calZ(k;\mu) := \sum_{m\in\Z} \frac{(-1)^m}{((2m+1)\pi - \mu)^{k+1}},
\end{equation}
in terms of the Apostol-Euler polynomials $\calE_k(x;\lambda)$. We have to avoid $\lambda = -1$, where they are not defined, so that we cannot take $\mu=\pi$ nor in general $\mu=(2l+1)\pi$ with $l \in \Z$;
this also avoids vanishing denominators in the series. Without loss of generality, we can assume that $\mu \in (-\pi,\pi)$.

The aim of this section is to prove the following result:

\begin{theorem}
\label{teo:summu}
For $\mu \in (-\pi,\pi)$ and $k=0,1,2,\dots$, we have
\begin{equation}
\label{eq:summu}
  \calZ(k;\mu) = \sum_{m\in\Z} \frac{(-1)^m}{((2m+1)\pi - \mu)^{k+1}}
  = \frac{1}{2 \cdot k!} \, i^k  e^{i\mu/2} \calE_{k}(1/2;e^{i\mu}),
\end{equation}
where $\calE_{k}(x;\lambda)$ are the Apostol-Euler polynomials.
\end{theorem}

For Apostol-Bernoulli polynomials, \cite[Proposition~1]{NRV} shows that for any $\lambda \in \C \setminus \{0,1\}$, $m \in \Z$ and $k \in \N$, we have
\begin{equation*}
\int_0^1 \lambda^x \calB_{k}(x;\lambda)e^{-2\pi imx}\,dx
= -\frac{k!}{(2\pi i m - \log \lambda)^{k}}.
\end{equation*}
Of course, behind the scenes, we are really calculating the Fourier coefficients in~\eqref{eq:FS-AB}, but we do not need to know this since we are purposefully avoiding Fourier analysis.

Similarly, for the Apostol-Euler polynomials (we are now computing the Fourier coefficients in~\eqref{eq:FS-AE}, but we do not need to know it), we have the following.

\begin{proposition}
\label{prop:FC-AE}
For any $\lambda \in \C \setminus \{0,-1\}$, $m \in \Z$ and $k \in \N \cup \{0\}$, we have
\begin{equation}
\label{eq:FC-AE}
\calI_{k,m} := \int_0^1 \lambda^x\calE_{k}(x;\lambda)e^{-(2m+1)\pi ix}\,dx
= \frac{2 \cdot k!}{((2m+1)\pi i - \log \lambda)^{k+1}}.
\end{equation}
\end{proposition}

\begin{proof}
Fixing $m$, we use induction on $k$. One easily checks the case $k=0$, noting that
$\calE_0(x;\lambda) = 2/(1+\lambda)$. Let $k \geq 1$. Assuming \eqref{eq:FC-AE} is true for $k-1$, integrate it by parts, using the derivative formula $\calE'_{k}(x;\lambda) = k\calE_{k-1}(x;\lambda)$ to
obtain
\[
\int_0^1 \lambda^x\calE_{k}(x;\lambda)e^{-(2m+1)\pi ix}\,dx
= \frac{\lambda\calE_{k}(1;\lambda)+\calE_{k}(0;\lambda)}{(2m+1)\pi i - \log \lambda}
+ \frac{2\cdot k!}{((2m+1)\pi i - \log \lambda)^{k+1}}.
\]
The proof is concluded by checking that $\lambda\calE_{k}(1;\lambda) + \calE_{k}(0;\lambda) = 0$ for $k \geq 1$. This follows from
\[
g(x,\lambda,z) \stackrel{\text{\tiny def}}{=} \frac{2 e^{xz}}{\lambda e^z+1} 
= \sum_{k=0}^\infty \calE_k(x;\lambda)\, \frac{z^k}{k!}
\]
noting that $\lambda g(1,\lambda,z) + g(0,\lambda,z) = 2$. 
\end{proof}

Now, let us prove Theorem~\ref{teo:summu}, proceeding as in the proof of Theorem~\ref{teo:beta2k+1} (we will not repeat all the details). Using the notation in~\eqref{eq:FC-AE}, we have
\[
  \frac{2\cdot k!}{i^{k+1}} \calZ(k;\mu) =  \sum_{m\in\Z} (-1)^m \calI_{k,m}.
\]

Define $e_{k}(x;\lambda) = \calE_{k}(x;\lambda)-\calE_{k}(1/2;\lambda)$, so that $e_{k}(1/2;\lambda)=0$. Note that
\begin{align*}
  & \int_0^1 \lambda^x \calE_{k}(1/2;\lambda) e^{-(2m+1)\pi ix}\,dx 
  = \calE_{k}(1/2;\lambda) \int_0^1 e^{(-(2m+1)\pi + \mu) ix}\,dx \\
  &\qquad= \calE_{k}(1/2;\lambda) \frac{e^{(-(2m+1)\pi + \mu) ix}}{(-(2m+1)\pi + \mu) i}\bigg|_{x=0}^1 
  = \calE_{k}(1/2;\lambda) \frac{e^{-\pi i + \mu i}-1}{(-(2m+1)\pi + \mu) i} \\
  &\qquad= \calE_{k}(1/2;\lambda) \frac{-\lambda-1}{(-(2m+1)\pi + \mu) i} 
  = \frac{(1+\lambda)\calE_{k}(1/2;\lambda)}{((2m+1)\pi-\mu) i},
\end{align*}
hence
\[
  \calI'_{k,m} := \int_0^1 \lambda^x e_{k}(x;\lambda) e^{-(2m+1)\pi ix} \,dx 
  = \calI_{k,m} + \frac{(1+\lambda)\calE_{k}(1/2;\lambda)}{(2m+1)\pi-\mu} \, i.
\]
Thus we have
\begin{align*}
  \frac{2\cdot k!}{i^{k+1}} \calZ(k;\mu) &=  \sum_{m\in\Z} (-1)^m \calI_{k,m} \\
  &= \sum_{m\in\Z} (-1)^m \calI'_{k,m} 
  - (1+\lambda)\calE_{k}(1/2;\lambda) i \sum_{m\in\Z} \frac{(-1)^m}{(2m+1)\pi-\mu}.
\end{align*}

We need to evaluate $\sum_{m\in\Z} \frac{(-1)^m}{(2m+1)\pi-\mu}$. To do so, consider instead
\[
  \sum_{m\in\Z} \frac{(-1)^m}{(2m+1)\pi-\mu} \,x^{(2m+1)\pi-\mu},
\] 
which plays a role analogous to~\eqref{eq:arctan}. Separating the sum into positive ($m \geq 0$) and negative ($m \leq -1$) powers, we obtain two series which when differentiated term-by-term yield geometric series which are easily summed. Now, integrating term-by-term (adequately choosing the integration constants) and letting $x=1$, we arrive at
\begin{align*}
  \sum_{m\in\Z} \frac{(-1)^m}{(2m+1)\pi-\mu}
  &= \frac{1}{4} \left(\cot \left(\frac{\mu +\pi }{4}\right) 
  - \cot \left(\frac{\mu -\pi}{4}\right)\right) \\
  &= \sin \left(\frac{\mu}{2}\right) \csc (\mu)
  = \frac{1}{2} \sec \left(\frac{\mu}{2}\right)
  = \frac{1}{2\cos(\mu/2)}.
\end{align*}

Now, we need to find a telescoping series that allows us to sum
\[
  \sum_{m\in\Z} (-1)^m \calI'_{k,m}
  = \sum_{m\in\Z} (-1)^m \int_0^1 \lambda^x e_{k}(x;\lambda) e^{-(2m+1)\pi ix} \,dx.
\]
To do this, it suffices to observe that
\[
  e^{-(2m+2)\pi ix} + e^{-(2m)\pi ix} = e^{-(2m+1)\pi ix} (e^{-\pi ix} + e^{\pi ix}) 
  = e^{-(2m+1)\pi ix} \cdot 2 \cos(\pi x),
\]
hence
\begin{equation}
\label{eq:telescospi}
  e^{-(2m+1)\pi ix} = \frac{e^{-(2m+2)\pi ix}}{2 \cos(\pi x)} + \frac{e^{-(2m)\pi ix}}{2 \cos(\pi x)}.
\end{equation}
This achieves exactly what we need thanks to the alternating sign $(-1)^{m}$. Furthermore, the denominator $\cos(\pi x)$ vanishes at $x=1/2$ which is also where $e_{k}(x;\lambda)$ vanishes, just as we want. Moreover, the sum of the series $\sum_{m\in\Z} (-1)^m \calI'_{k,m}$ is~$0$, as can be seen by again applying the ``weak'' version of the Riemann-Lebesgue lemma (Lemma~\ref{lem:RL}) that we also used in Theorems~\ref{teo:zeta2k} and~\ref{teo:beta2k+1}.

Putting everything together, we obtain
\[
  \frac{2\cdot k!}{i^{k+1}} \calZ(k;\mu) = 
  - (1+\lambda)\calE_{k}(1/2;\lambda) i \sum_{m\in\Z} \frac{(-1)^m}{(2m+1)\pi-\mu}
  = \frac{- (1+\lambda)\calE_{k}(1/2;\lambda) i}{2 \cos(\mu/2)},
\]
and so
\[
  \calZ(k;\mu) = \frac{(1+\lambda)\calE_{k}(1/2;\lambda) }{4 \cdot k!\, \cos(\mu/2)}\,i^k
  = \frac{\lambda^{-1/2}+\lambda^{1/2}}{4 \cdot k!\, \cos(\mu/2)}\, i^k \lambda^{1/2}\calE_{k}(1/2;\lambda).
\]
Moreover, $\lambda^{-1/2}+\lambda^{1/2} = e^{-\mu i/2} + e^{\mu i/2} = 2 \cos(\mu/2)$, so we have proved Theorem~\ref{teo:summu}.

\subsection{Another expression for the sum}

As desired, Theorem~\ref{teo:summu} gives the value of $\calZ(k,m)$ in terms of the Apostol-Euler polynomials. But $\mu$ is assumed real while the sum in Theorem~\ref{teo:summu} has complex coefficients. It is interesting to try to find an alternative formula avoiding complex numbers. With this in mind, we have the following:

\begin{lemma}
\label{lem:Ekmu}
Let $\calE_k(x;\lambda)$ denote the Apostol-Euler polynomials as defined in~\eqref{eq:GF-AE}, with $\lambda = e^{i\mu}$, $\mu \in (-\pi,\pi)$. Then
\begin{equation*}
  \bfE _k(\mu) := i^k e^{i\mu/2} \calE_k(1/2;e^{i\mu}), \qquad k=0,1,2,\dots,
\end{equation*}
is always a real number. In fact, the values $\bfE_k(\mu)$ may be computed via the generating function
\begin{equation}
\label{eq:Ekmu}
  \frac{1}{\cos(w/2)} = \sum_{k=0}^{\infty} \bfE_k(\mu) \frac{(w-\mu)^k}{k!},
  \qquad |w - \mu| < \min |\mu \pm \pi| 
\end{equation}
and hence
\[
  \bfE_k(\mu) = \frac{d^k}{d\mu^k} \left(\frac{1}{\cos(\mu/2)}\right).
\]
\end{lemma}

\begin{proof}
Multiply both sides of \eqref{eq:GF-AE} (with $\lambda = e^{i\mu}$) by $e^{i\mu/2}$, set $x=1/2$ and $z = is$. Then we have
\[
  \frac{2 e^{(s+\mu)i/2}}{e^{(s+\mu)i}+1}
  = \sum_{k=0}^\infty \bfE_k(\mu) \frac{s^k}{k!}.
\]
Substituting $w=s+\mu$, this becomes
\[
  \frac{2 e^{wi/2}}{e^{wi}+1}
  = \sum_{k=0}^\infty \bfE_k(\mu) \frac{(w-\mu)^k}{k!},
\]
and the result follows if we show that $2 e^{wi/2}/(e^{wi}+1)$ is real when~$w$ is. 
This is easy:
\[
  \frac{2 e^{wi/2}}{e^{wi}+1}
  = \frac{2}{e^{-wi/2}(e^{wi}+1)}
  = \frac{2}{e^{wi/2}+e^{-wi/2}}
  = \frac{1}{\cos(w/2)},
\]
using Euler's formula $\cos\theta = (e^{i\theta}+e^{-i\theta})/2$.
\end{proof}



Then, as a consequence of Theorem~\ref{teo:summu} and Lemma~\ref{lem:Ekmu}, we have the following:

\begin{corollary}
\label{cor:summu}
For $\mu \in (-\pi,\pi)$ and $k=0,1,2,\dots$, we have
\[
  \calZ(k;\mu) = \sum_{m\in\Z} \frac{(-1)^m}{((2m+1)\pi - \mu)^{k+1}}
  = \frac{\bfE_k(\mu)}{2 \cdot k!}
\]
where the values $\bfE_k(\mu)$ are the coefficients of the generating function~\eqref{eq:Ekmu}.
\end{corollary}

\begin{remark}
Requiring $\mu \in (-\pi,\pi)$ is not essential, we could just as well assume $\mu \in \R \setminus \{(2m+1)\pi : m\in\Z\}$ and the formula $\calZ(k;\mu) = \bfE_k(\mu)/(2 \cdot k!)$ will still be valid. In particular, observe that $\calZ(k;\mu+2\pi) = -\calZ(k;\mu)$ (in the series this corresponds to $m \mapsto m-1$) and $\bfE_k(\mu+2\pi) = -\bfE_k(\mu)$ (in the generating function, $\cos((\mu+2\pi)/2) = -\cos(\mu/2)$).
\end{remark}

\begin{remark}
Once we have proved that
\[
  \calZ(0;\mu) = \sum_{m\in\Z} \frac{(-1)^m}{(2m+1)\pi - \mu}
  = \frac{\bfE_0(\mu)}{2} = \frac{1}{2\cos(w/2)},
\]
the general formula
\[
  \calZ(k;\mu) = \sum_{m\in\Z} \frac{(-1)^m}{((2m+1)\pi - \mu)^{k+1}}
  = \frac{\bfE_k(\mu)}{2 \cdot k!}
\]
follows by induction in $k$ via termwise differentiation with respect to $\mu$, since $\bfE_k(\mu) = \frac{d}{d\mu} \bfE_{k-1}(\mu)$.
\end{remark}

The first values of $\calZ(k;\mu)$ are as follows:
\begin{align*}
\calZ(0;\mu) &=
\frac{1}{2\cos(\frac{\mu}{2})},
\qquad
\calZ(1;\mu) =
\frac{\sin(\frac{\mu}{2})}{4\cos^2(\frac{\mu}{2})},
\qquad
\calZ(2;\mu) =
\frac{3 - \cos(\mu)}{32 \cos^3(\frac{\mu}{2})},
\\
\calZ(3;\mu) &=
\frac{23 \sin(\frac{\mu}{2}) - \sin(\frac{3\mu}{2})}{384 \cos^4(\frac{\mu}{2})},
\qquad
\calZ(4;\mu) =
\frac{115 - 76 \cos(\mu) + \cos(2\mu)}{6144 \cos^5(\frac{\mu}{2})},
\\
\calZ(5;\mu) &=
\frac{1682 \sin(\frac{\mu}{2}) - 237 \sin(\frac{3\mu}{2}) + \sin(\frac{5\mu}{2})}
  {122880 \cos^6(\frac{\mu}{2})},
\\
\calZ(6;\mu) &=
\frac{11774 - 10543 \cos(\mu) + 722 \cos(2\mu) - \cos(3\mu)}{2949120 \cos^7(\frac{\mu}{2})}.
\end{align*}
Except for $k=0$, the constant in the denominator is $2^{2k}k!$.

\section{A partial fraction expansion}
\label{sec:parfrac}

In 1748, in \S\,178 of his \textit{Introductio in Analysin Infinitorum} (see \cite{Eu-fac}), Euler proved that
\begin{equation}
\label{eq:sumtheta1}
\sum_{n\in\Z} \frac{1}{n+\theta} 
= \lim_{N \to \infty} \sum_{n=-N}^N \frac{1}{n+\theta} = \frac{\pi}{\tan(\pi\theta)},
\qquad \theta \in \R \setminus \Z.
\end{equation}
This is the partial fraction expansion of the cotangent function, and it is a remarkable result that requires defining the convergence of the series as the symmetric limit $\lim_{N \to \infty} \sum_{n=-N}^N$ (equivalently, grouping pairs of summands corresponding to opposite indices $\pm n$) because the series is not absolutely convergent. A detailed ``traditional'' proof can be seen, for instance, in \cite[\S\,24, equation 117]{Kn}; and, for a nice self-contained proof with detailed explanations, see~\cite{Hof}. Another kind of elementary proofs use the so-called Herglotz trick; this can be seen, for instance, in~\cite{AiZi-cot}.

It is clear that, with small adjustments, this series becomes
\begin{equation}
\label{eq:summu1}
\sum_{m\in\Z} \frac{1}{2m\pi-\mu} 
= \lim_{M \to \infty} \sum_{m=-M}^M \frac{1}{2m\pi-\mu} = \frac{-1}{2\tan(\mu/2)}.
\end{equation}
Our goal here is to give the value of the expansions~\eqref{eq:summus} for integer $s=k\ge2$ in terms of the Apostol-Euler polynomials. For this purpose, we will see in the next section that it is enough to start with
\begin{equation*}
\sum_{m\in\Z} \frac{1}{(2m\pi-\mu)^2}, 
\end{equation*}
whose value can be obtained from the following result:

\begin{proposition}
\label{prop:sumtheta2}
For $\theta \in \R \setminus \Z$, we have
\begin{equation}
\label{eq:sumtheta2}
\sum_{n\in\Z} \frac{1}{(n+\theta)^2} = \frac{\pi^2}{\sin^2(\pi\theta)}.
\end{equation}
\end{proposition}

Actually, \eqref{eq:sumtheta2} can be obtained form \eqref{eq:sumtheta1} by differentiating term by term. However, instead of proving~\eqref{eq:sumtheta1} first and then doing that, here we directly prove~\eqref{eq:sumtheta2} by adapting the proof of \eqref{eq:sumtheta1} in \cite{AiZi-cot}, also using the Herglotz trick. By using the previously determined value $\sum_{n=1}^{\infty} 1/n^2 = \pi^2/6$ in our proof, in Lemma~\ref{lem:fg}, the treatment of \eqref{eq:sumtheta2} is simpler than that required for \eqref{eq:sumtheta1}, in part because the series converges absolutely. Thus, let us now give a proof of Proposition~\ref{prop:sumtheta2} using elementary methods.

Define the functions
\begin{equation}
\label{eq:deffg}
  f(\theta) = \frac{\pi^2}{\sin^2(\pi\theta)},
  \qquad 
  g(\theta) = \sum_{n\in\Z} \frac{1}{(n+\theta)^2},
  \qquad \theta \in \R \setminus \Z.
\end{equation}
We want to prove that $f=g$. It is clear that $f$ and $g$ are even, and both are periodic of period~$1$. Let us see that they share some other properties.

\begin{lemma}
\label{lem:fg}
The functions $f$ and $g$ defined in \eqref{eq:deffg} satisfy
\begin{equation}
\label{eq:eqfunfg}
  f\Big(\frac{\theta}{2}\Big) + f\Big(\frac{\theta+1}{2}\Big) = 4f(\theta),
  \qquad
  g\Big(\frac{\theta}{2}\Big) + g\Big(\frac{\theta+1}{2}\Big) = 4g(\theta),
\end{equation}
and
\begin{equation}
\label{eq:lim0fg}
  \lim_{\theta\to0} \Big( f(\theta) - \frac{1}{\theta^2} \Big)
  = \lim_{\theta\to0} \Big( g(\theta) - \frac{1}{\theta^2} \Big)
  = \frac{\pi^2}{3}.
\end{equation}
\end{lemma}

\begin{proof}
Using the identity
\begin{align*}
  \frac{1}{\sin^2(\pi\theta/2)} + \frac{1}{\sin^2(\pi(\theta+1)/2)}
  &= \frac{1}{\sin^2(\pi\theta/2)} + \frac{1}{\cos^2(\pi\theta/2)} \\
  &= \frac{\cos^2(\pi\theta/2) + \sin^2(\pi\theta/2)}{\sin^2(\pi\theta/2)\cos^2(\pi\theta/2)}
  = \frac{4}{\sin^2(\pi\theta)}
\end{align*}
proves that~$f$ satisfies~\eqref{eq:eqfunfg}.
We check that $g$ satisfies the same functional equation. We have $g(\theta) = \lim_{N\to\infty} g_N(\theta)$, where
\[
  g_N(\theta) = \frac{1}{x^2} 
  + \sum_{n=1}^{N} \Big( \frac{1}{(\theta+n)^2} + \frac{1}{(\theta-n)^2} \Big).
\]
Since
\[
  \frac{1}{(\theta/2 \pm n)^2} + \frac{1}{((\theta+1)/2 \pm n)^2}
  = \frac{4}{(\theta \pm 2n)^2} + \frac{4}{(\theta \pm 2n+1)^2},
\]
it is clear that
\[
  g_N\Big(\frac{\theta}{2}\Big) + g_N\Big(\frac{\theta+1}{2}\Big)
  = 4g_{2N}(\theta) + \frac{4}{(\theta+2N+1)^2},
\]
and letting $N\to\infty$ we obtain \eqref{eq:eqfunfg} for~$g$.

We now turn to~\eqref{eq:lim0fg}.
For $f$, using the Taylor expansion $\sin(x) = x - x^3/3! + x^5/5! - \cdots$, we have
\[
  \lim_{\theta\to0} \Big( f(\theta) - \frac{1}{\theta^2} \Big)
  = \lim_{\theta\to0} \Big( \frac{\pi^2}{\sin^2(\pi\theta)} - \frac{1}{\theta^2} \Big)
  = \lim_{\theta\to0} \frac{\pi^2\theta^2 - \big(\pi\theta - (\pi\theta)^3/3! 
  + \cdots\big)^2}{\theta^2 \sin^2(\pi\theta)}
  = \frac{\pi^2}{3}.
\]
For~$g$, 
\[
  \lim_{\theta\to0} \Big( g(\theta) - \frac{1}{\theta^2} \Big) 
  = \sum_{n\in\Z\setminus\{0\}} \lim_{\theta\to0} \frac{1}{(n+\theta)^2}
  = 2 \sum_{n=1}^{\infty} \frac{1}{n^2} = \frac{\pi^2}{3},
\]
and the proof is concluded (it is easy to justify that we can introduce the limit inside the summation using, for example, uniform convergence on the interval $[-1/2,1/2]$, but we are not going to worry about the details).
\end{proof}

Now, consider the difference
\begin{equation*}
  h(\theta) = f(\theta) - g(\theta) 
  = \big(f(\theta) - 1/\theta^2\big)  - \big(g(\theta) - 1/\theta^2\big).
\end{equation*}
The corresponding properties of $f$ and $g$ imply that $h$ is periodic of period $1$
and satisfies
\begin{equation}
\label{eq:eqfunf}
  h\Big(\frac{\theta}{2}\Big) + h\Big(\frac{\theta+1}{2}\Big)
  = 4h(\theta).
\end{equation}
Moreover, $\lim_{\theta\to0} h(\theta) = 0$, thus $h$ becomes continuous at $\theta=0$ by defining $h(0)=0$. In fact, by $1$-periodicity, $h$ extends to a continuous function on $\R$ by defining $h(k)=0$ for $k\in\Z$.

Since $h$ is a continuous periodic function, it has a maximum $m \in \R$; let $\theta_0 \in [0,1]$ such that $h(\theta_0) = m$. Because $h(0)=0$, we have $m \ge0$. Using~\eqref{eq:eqfunf},
\[
  h\Big(\frac{\theta_0}{2}\Big) + h\Big(\frac{\theta_0+1}{2}\Big)
  = 4h(\theta_0) = 4m,
\]
but then either $h(\theta_0/2) \ge 2m$ or $h((\theta_0+1)/2) \ge 2m$. Because $m$ is the maximum, this can only happen if $m=0$. Consequently, $m=0$ and $h = f - g =0$, which concludes the proof of Proposition~\ref{prop:sumtheta2}.

\section{Apostol-Euler polynomials to sum a ``positive series''}
\label{sec:musinalt}

For $\lambda = e^{\mu i}$, $\mu \in \R$, we wish to find the value of the sum
\begin{equation}
\label{eq:calZtilde}
  \widetilde{\calZ}(k;\mu) := \sum_{m\in\Z} \frac{1}{(2m\pi - \mu)^{k+1}}.
\end{equation}
The case $k=1$ follows from Proposition~\ref{prop:sumtheta2}:
\begin{equation}
\label{eq:summu2}
  \widetilde{\calZ}(1;\mu) = \sum_{m\in\Z} \frac{1}{(2m\pi-\mu)^2} = \frac{1}{4\sin^2(\mu/2)}.
\end{equation}
We are calling the series $\widetilde{\calZ}(k;\mu)$ in~\eqref{eq:calZtilde} a ``positive series'' only as a means of contrasting it with the terminology of ``alternating series'' as refers to the series $\calZ(k;\mu)$ defined in~\eqref{eq:calZ}. Clearly the summands in \eqref{eq:calZtilde} are not always positive.

The goal of this section is to prove the following:

\begin{theorem} 
\label{teo:summusa}
For $\mu \in \R \setminus 2\pi\Z$ and $k = 1,2,\dots$, we have
\begin{equation}
\label{eq:summusa}
  \widetilde{\calZ}(k;\mu) 
   = \frac{1}{2 \cdot k!} \, i^{k+1} e^{i\mu} \calE_{k}(1;-e^{i\mu}).
\end{equation}
\end{theorem}

\begin{remark}
Before we turn to the proof, note that if $\mu \in 2\pi\Z$, not only does one of the denominators in the series $\widetilde{\calZ}(k;\mu)$ vanish, but also, since $-e^{\mu i} = -1$, the Apostol-Euler polynomials $\calE_{k}(x;-e^{-i\mu})$ are not even defined. Thus \emph{both} sides of~\eqref{eq:summusa} are meaningless.

Observe also that~\eqref{eq:summusa} is not valid for $k=0$. Indeed, we have seen in~\eqref{eq:summu1} that $\widetilde{\calZ}(0;\mu) = \frac{-1}{2\tan(\mu/2)}$. However, $\calE_{0}(1;-\lambda) = 2/(1-\lambda)$, hence the right hand side of \eqref{eq:summusa} is
$- 2i e^{i\mu}/(1-e^{i\mu}) = i + \frac{1}{\tan(\mu/2)}$. 
Since our goal is to express $\widetilde{\calZ}(k;\mu)$ in terms of the Apostol-Euler polynomials, this exception is the reason why, in Section~\ref{sec:parfrac}, we proved \eqref{eq:sumtheta2} instead of~\eqref{eq:sumtheta1}.
\end{remark}

To prove Theorem~\ref{teo:summusa}, one could give a complete proof using Apostol-Bernoulli polynomials, but since they are essentially the same as the Apostol-Euler polynomials (recall~\eqref{eq:relAB-AE}), it is not really worth the trouble. Instead, we adapt the proof of Theorem~\ref{teo:summu}, changing $\lambda$ to $-\lambda$ and seeing where this leads.

Keeping in mind that $\lambda = e^{i\mu}$, we have $-\lambda = e^{i(\mu+\pi)}$, hence $\log(-\lambda) = \pi i + \log\lambda$, and so by Proposition~\ref{prop:FC-AE}, we have
\begin{equation}
\label{eq:FC-AEpares}
\begin{aligned}
\widetilde{\calI}_{k,m} &:= \int_0^1 (-\lambda)^x \calE_{k}(x;-\lambda)e^{-(2m+1)\pi ix}\,dx
= \frac{2 \cdot k!}{((2m+1)\pi i - \log (-\lambda))^{k+1}} \\
&= \frac{2 \cdot k!}{(2m\pi i - \log\lambda)^{k+1}}.
\end{aligned}
\end{equation}
Since $\log \lambda = i \mu$, we have
\[
  \frac{2\cdot k!}{i^{k+1}} \widetilde{\calZ}(k;\mu) =  \sum_{m\in\Z} \widetilde{\calI}_{k,m}.
\]

Now, we want to find a function $\widetilde{e}_{k}(x;-\lambda)$ straightforwardly related to $\calE_{k}(x;-\lambda)$ but vanishing at $0$ and~$1$. The reason for this lies with the telescoping series~\eqref{eq:telessinpi} used below, which is analogous to~\eqref{eq:telescospi} but instead of $\cos(\pi x)$, introduces a factor of $\sin(\pi x)$ in the denominator, which vanishes at $x = 0$, $1$. Since we need  $\widetilde{e}_{k}(x;-\lambda)/\sin(\pi x)$ to be continuously differentiable in order
to apply Lemma~\ref{lem:RL}, the numerator $\widetilde{e}_{k}(x;-\lambda)$ must also vanish at these points.

The following expression related to $\calE_{k}(x;-\lambda)$ satisfies our requirements:
\begin{align*}
  \widetilde{e}_{k}(x;-\lambda) 
  &= \calE_{k}(x;-\lambda) - \calE_{k}(0;-\lambda) - (1-\lambda) x \calE_{k}(1;-\lambda) \\
  &= \calE_{k}(x;-\lambda) - \Big(\calE_{k}(0;-\lambda) + (1-\lambda) x \calE_{k}(1;-\lambda)\Big)
\end{align*}
(from the proof of Proposition~\ref{prop:FC-AE} we have that $(-\lambda)\calE_{k}(1;-\lambda) + \calE_{k}(0;-\lambda) = 0$ for $k \ge1$).

Since $\calE_{k}(0;-\lambda) = \lambda \calE_{k}(1;-\lambda)$, we may rewrite
$\calE_{k}(0;-\lambda) + (1-\lambda) x \calE_{k}(1;-\lambda) = (\lambda+(1-\lambda)x) \calE_{k}(1;-\lambda)$.
Furthermore, $\calE_{1}(x;-\lambda) = 2\lambda/(1-\lambda)^2 + 2x/(1-\lambda)$, therefore $\frac{(1-\lambda)^2}{2} \calE_{1}(x;-\lambda) = \lambda+(1-\lambda)x$, and thus
\begin{equation}
\label{eq:ektilde}
  \widetilde{e}_{k}(x;-\lambda) 
  = \calE_{k}(x;-\lambda)  - \frac{(1-\lambda)^2}{2} \calE_{k}(1;-\lambda) \calE_{1}(x;-\lambda).
\end{equation}

If we let $k=1$ in~\eqref{eq:FC-AEpares}, we obtain
\[
  \int_0^1 (-\lambda)^x \calE_{1}(x;-\lambda) e^{-(2m+1)\pi ix}\,dx
  = \frac{2 \cdot 1!}{(2m\pi i - \log \lambda)^{2}} = \frac{-2}{(2m\pi - \mu)^{2}}.
\]
Hence by~\eqref{eq:ektilde}, we have
\[
  \widetilde{\calI}'_{k,m} := \int_0^1 (-\lambda)^x \widetilde{e}_{k}(x;-\lambda) e^{-(2m+1)\pi ix} \,dx 
  = \widetilde{\calI}_{k,m} + \frac{(1-\lambda)^2\calE_{k}(1;-\lambda)}{(2m\pi-\mu)^2}.
\]
Thus
\[
  \frac{2\cdot k!}{i^{k+1}} \widetilde{\calZ}(k;\mu)
  = \sum_{m\in\Z} \widetilde{\calI}_{k,m}
  = \sum_{m\in\Z} \widetilde{\calI}'_{k,m} 
    - (1-\lambda)^2 \calE_{k}(1;-\lambda) \sum_{m\in\Z} \frac{1}{(2m\pi-\mu)^2}.
\]
The series on the right has already been computed in~\eqref{eq:summu2}.
All we need to do now is to express $\sum_{m\in\Z} \widetilde{\calI}'_{k,m}$ as a telescoping series and verify that it sums to~$0$. 

Assuming this has been done, we obtain, for $k \geq 1$,
\begin{align*}
  \widetilde{\calZ}(k;\mu) 
  &= -\frac{i^{k+1}}{2\cdot k!} (1-\lambda)^2 \calE_{k}(1;-\lambda) 
    \sum_{m\in\Z} \frac{1}{(2m\pi-\mu)^2} \\
  &= -\frac{i^{k+1}}{2 \cdot k!} (\lambda^{-1/2}-\lambda^{1/2})^2 \lambda \,
    \frac{\calE_{k}(1;-\lambda)}{4 \sin^2(\mu/2)} 
  = -\frac{(e^{-i\mu/2}-e^{i\mu/2})^2}{8 \cdot k! \sin^2(\mu/2)} 
    \, i^{k+1} e^{i\mu} \calE_{k}(1;-\lambda) \\
  &= -\frac{(-2i\sin(\mu/2))^2}{8 \cdot k! \sin^2(\mu/2)} 
    \, i^{k+1} e^{i\mu} \calE_{k}(1;-\lambda) 
  = \frac{1}{2 \cdot k!} \, i^{k+1} e^{i\mu} \calE_{k}(1;-e^{i\mu}),
\end{align*}
which completes the proof of Theorem~\ref{teo:summusa}.
It is worth mentioning that the previous argument is not valid for $k = 0$, because the
relation $(-\lambda)\calE_{k}(1;-\lambda) + \calE_{k}(0;-\lambda) = 0$ does not hold in that case.


As for expressing the $\sum_{m\in\Z} \widetilde{\calI}'_{k,m}$ as a telescoping series, noting that we lack the sign $(-1)^m$ which was present in \eqref{eq:calZ}, we instead consider
\[
  e^{-(2m+2)\pi ix} - e^{-(2m)\pi ix} = e^{-(2m+1)\pi ix} (e^{-\pi ix} - e^{\pi ix}) 
  = e^{-(2m+1)\pi ix} \cdot (-2 i) \sin(\pi x),
\]
which gives, in analogy with~\eqref{eq:telescospi},
\begin{equation}
\label{eq:telessinpi}
  e^{-(2m+1)\pi ix} = \frac{ie^{-(2m+2)\pi ix}}{2\sin(\pi x)} 
  - \frac{ie^{-2m\pi ix}}{2\sin(\pi x)}.
\end{equation}
This now allows the telescoping argument to work. As we mentioned above, the denominator $\sin(\pi x)$ vanishes at $x=0$ and $x=1$, as does $\widetilde{e}_{k}(x;-\lambda)$, which implies that the sum of the series $\sum_{m\in\Z} \widetilde{\calI}'_{k,m}$ is~$0$, using Lemma~\ref{lem:RL} as we did previously in the proofs of Theorems~\ref{teo:zeta2k}, \ref{teo:beta2k+1} and~\ref{teo:summu}.

With all this said, the proof of Theorem~\ref{teo:summusa} is concluded.

\subsection{Another expression for the sum}

As was the case with \eqref{eq:summu} in Theorem~\ref{teo:summu}, looking at the expression \eqref{eq:summusa}, it is not immediately obvious that the values $\widetilde{\calZ}(k;\mu)$ are real, as must be the case when $\mu$ is real. However, we can find an alternative formula for $\widetilde{\calZ}(k;\mu)$ which solves this problem. Let us begin with a result analogous to Lemma~\ref{lem:Ekmu}:

\begin{lemma}
\label{lem:Ekmutilde}
Let $\calE_k(x;-\lambda)$ be the Apostol-Euler polynomials, defined by~\eqref{eq:GF-AE}, with parameter $-\lambda = e^{i(\mu+\pi)}$, $\mu \in \R \setminus 2\pi\Z$ (recall that $-\lambda=-1$ is not allowed since $\calE_k(x;-1)$ is undefined). Then 
\begin{equation*}
  \widetilde{\bfE}_k(\mu)  
  := i^{k+1} e^{i\mu} \calE_k(1;-e^{i\mu}),
  \qquad k=1,2,3,\dots,
\end{equation*}
is a real number (this is false for $k=0$, since in that case the expression on the right is $2i\lambda/(1-\lambda) = -1/\tan(\mu/2)-i$, which is not real). In fact, the values $\widetilde{\bfE}_k(\mu)$ may be computed via the generating function
\begin{equation}
\label{eq:Ekmutilde}
  -\frac{1}{\tan(w/2)}   
  = -\frac{1}{\tan(\mu/2)} + \sum_{k=1}^{\infty} \widetilde{\bfE}_k(\mu) \frac{(w-\mu)^k}{k!},
  \qquad |w - \mu| < d(\mu, 2\pi\Z), 
\end{equation}
and hence, if for $k = 0$ we define $\widetilde{\bfE}_0(\mu) = -\frac{1}{\tan(\mu/2)}$, then
\[
  \widetilde{\bfE}_k(\mu) = -\frac{d^k}{d\mu^k} \left(\frac{1}{\tan(\mu/2)}\right),
  \qquad k=0,1,2,3,\dots. 
\]
\end{lemma}

\begin{proof} 
In~\eqref{eq:GF-AE}, first change $\lambda$ to $-\lambda$; next, multiply by $i\lambda$ and then substitute $\lambda = e^{i\mu}$, $x = 1$ and $z = is$. This yields
\begin{equation}
\label{eq:auxdemotilde}
  \frac{2i e^{i (\mu + s)}}{1 - e^{i (\mu + s)}}
  = \sum_{k=0}^\infty i e^{i\mu} \calE_k(1,-e^{i\mu}) \frac{(is)^k}{k!}
  = C_0(\mu) + \sum_{k=1}^\infty \widetilde{\bfE}_k(\mu) \frac{s^k}{k!}
\end{equation}
where, either setting $s = 0$ or recalling that $\calE_{0}(x; \lambda) = \frac{2}{1 + \lambda}$, we have
\[
  C_0(\mu) = i e^{i\mu} \calE_{0}(1,-e^{i\mu})
  = \frac{2 i e^{i \mu}}{1 - e^{i \mu}}.
\] 
Substituting $w = s + \mu$, we can rewrite~\eqref{eq:auxdemotilde} as a Taylor series for~$C_0$:
\begin{equation}
\label{eq:C0}
  C_0(w) = C_0(\mu) +  \sum_{k=1}^\infty \widetilde{\bfE}_k(\mu) \frac{(w-\mu)^k}{k!}.
\end{equation}
Finally, we compute
\[
  C_0(w) = \frac{2i e^{wi}}{1 - e^{wi}}
  = \frac{2 i e^{i w/2}}{e^{-iw/2} - e^{iw/2}}
  = -\frac{\cos(w/2) + i \sin(w/2)}{\sin(w/2)}
  = -\frac{1}{\tan(w/2)} - i.
\]
Substituting in~\eqref{eq:C0} and canceling the $-i$ from both sides of the equality, we obtain~\eqref{eq:Ekmutilde}.
\end{proof}

\begin{corollary}
\label{cor:summusa}
For $\mu \in \R \setminus 2\pi\Z$ and $k = 1,2,\dots$, we have
\[
  \widetilde{\calZ}(k;\mu) = \sum_{m\in\Z} \frac{1}{(2m\pi - \mu)^{k+1}}
  = \frac{\widetilde{\bfE}_k(\mu)}{2 \cdot k!}
\]
where the $\widetilde{\bfE}_k(\mu)$ are the coefficients of the generating function
\begin{equation*}
  -\frac{1}{\tan(w/2)} = \sum_{k=0}^{\infty} \widetilde{\bfE}_k(\mu) \frac{(w-\mu)^k}{k!},
  \qquad |w - \mu| < d(\mu, 2\pi\Z), 
\end{equation*}
and hence
\[
  \widetilde{\bfE}_k(\mu) = -\frac{d^k}{d\mu^k} \left(\frac{1}{\tan(\mu/2)}\right).
\]
\end{corollary}

The first $\widetilde{\calZ}(k;\mu)$ are as follows:
\begin{align*}
\widetilde{\calZ}(1;\mu) &=
\frac{1}{4\sin^2(\frac{\mu}{2})},
\qquad
\widetilde{\calZ}(2;\mu) =
\frac{-\cos(\frac{\mu}{2})}{8\sin^3(\frac{\mu}{2})},
\qquad
\widetilde{\calZ}(3;\mu) =
\frac{2 + \cos(\mu)}{48 \sin^4(\frac{\mu}{2})},
\\
\widetilde{\calZ}(4;\mu) &=
\frac{-11 \cos(\frac{\mu}{2}) - \cos(\frac{3\mu}{2})}{384 \sin^5(\frac{\mu}{2})},
\qquad
\widetilde{\calZ}(5;\mu) =
\frac{33 + 26 \cos(\mu) +  \cos(2\mu)}
  {3840 \sin^6(\frac{\mu}{2})},
\\
\widetilde{\calZ}(6;\mu) &=
\frac{-302 \cos(\frac{\mu}{2}) - 57 \cos(\frac{3\mu}{2}) - \cos(\frac{5\mu}{2})}
{46080 \sin^7(\frac{\mu}{2})},
\\
\widetilde{\calZ}(7;\mu) &=
\frac{1208 + 1191 \cos(\mu) + 120 \cos(2\mu) + \cos(3\mu)}
{645120 \sin^8(\frac{\mu}{2})}
\end{align*}
(except for $k=1$, the constant in the denominator is $2^{k}k!$).

\paragraph*{Acknowledgement.}
The research of the authors is supported by grant PID2021-124332NB-C22
(Mi\-nis\-te\-rio de Cien\-cia e Inno\-va\-ci\'on-Agen\-cia Esta\-tal de Inves\-ti\-ga\-ci\'on).



\end{document}